\documentclass{amsart}
\usepackage{hyperref}

\renewcommand{\l}{\ell}
\newcommand{\N}{\mathbb N}
\newcommand{\NN}{\mathbb N}
\newcommand{\R}{\mathbb R}
\newcommand{\RR}{\mathbb R}
\newcommand{\II}{\mathcal I}

\newcommand{\diam}{\operatorname{diam}}
\newcommand{\dist}{\operatorname{dist}}
\newcommand{\spa}{\operatorname{span}}
\renewcommand{\H}{\mathcal H^1}

\newcommand\eps{\varepsilon}
\newcommand\al{\alpha}

\newcommand\ga{\gamma}

\newtheorem{theorem}{Theorem}[]
\newtheorem{proposition}[theorem]{Proposition}
\newtheorem{lemma}[theorem]{Lemma}

\theoremstyle{remark}

\begin{document}
\title[Failure of the projection theorem in Banach spaces]{The Besicovitch-Federer projection theorem is false in every infinite dimensional Banach space}
\author{David Bate \and Marianna Cs\"ornyei \and Bobby Wilson}
\date{\today}
\address{Department of Mathematics, The University of Chicago, 5734 S. University Avenue, Chicago, IL 60637, USA}
\email{bate@math.uchicago.edu}
\email{csornyei@math.uchicago.edu}
\email{bobbylew@math.uchicago.edu}

\begin{abstract}
We construct a purely unrectifiable set of finite $\H$-measure in every infinite dimensional separable Banach space $X$ whose image under every $0\neq x^*\in X^*$ has positive Lebesgue measure.  This demonstrates completely the failure of the Besicovitch-Federer projection theorem in infinite dimensional Banach spaces.
\end{abstract}
\maketitle

\section{Introduction}


In a metric space $X$, a set $E\subset X$ is called \emph{rectifiable} if it can be covered, up to an $\H$-negligible set, by a countable family of Lipschitz images of $\RR$. A set is \emph{purely unrectifiable}, if it meets every Lipschitz image (or, equivalently, it meets every rectifiable set) in $\H$-measure zero.  Here and throughout this paper, $\H$ represents the 1-dimensional Hausdorff measure on $X$. For information about rectifiable and purely unrectifiable sets in a general metric space, see \cite{kirchheim}.

If $X=\R^n$, or more generally, if $X$ is a Banach space that admits the Radon-Nikodym property RNP, then these definitions have some pliability. By definition (see \cite{b-lindenstrauss}), $X$ admits RNP if each Lipschitz $f:\,\R\to X$ is differentiable at almost every point in $\R$.  It follows that there is an equivalent description of rectifiable sets with finite $\H$ measure: the sets that are rectifiable are those sets, $E \subset X$, that admit an \emph{approximate tangent} at $\H$-almost every of their points.  An approximate tangent at a point, $x$, is defined with respect to a subset $F$ for which $x$ is a \emph{density point}, i.e. for which $\H(F\cap B(x,r))/2r\to 1$ as $r\to 0$. A set $E$ has an \emph{approximate tangent}, $\theta$, at $x$, if there is a subset $F\subset E$ for which $x$ is a density point, and for every $\{x_n\} \subset F$ such that $x_n\rightarrow x$, $(x_n-x)/\|x_n-x\| \rightarrow \theta$. 

Furthermore, we observe that, in $\R^n$, the existence of an approximate tangent for $E$ implies that the projection of $E$ onto a line with direction not perpendicular to the approximate tangent direction has positive measure. 
A fundamental result of geometric measure theory is the Besicovitch-Federer projection theorem \cite{mattila} which characterizes pure unrectifiability in terms of projections.  The projection theorem states that for any set $E \subset \mathbb{R}^n$ with $\H(E)<\infty$, the projection of $E$ in almost every direction has measure zero \emph{if and only if} $E$ is purely unrectifiable.



It is a natural question to ask whether the projection theorem is also true in Banach spaces.  An immediate problem when formulating such a question is the nonexistence of an invariant probability measure on the set of projections in infinite dimensions (by a {\it projection} of course we mean the image of our set under a $0\neq x^*\in X^*$).  However, there are several notions of ``null set'', which suffice for such a formulation, see \cite{b-lindenstrauss}.

Indeed, in \cite{depauw}, De Pauw shows that the projection theorem fails in $\ell_2$, when considering \emph{Aronszajn-null} sets. He accomplishes this by constructing a purely unrectifiable set, $E \subset \ell_2$, with $\H(E)<\infty$, and a \emph{cube} $C$ in $\ell_2^*$ (of the form $C=\{x_0^*+\sum c_ix_i^*~:~ 0\le c_i\leq 1\}$) such that the image of $E$ under any $x^* \in C$ has positive measure.  The cube has positive measure with respect to a cube measure on $\ell_2$ which implies that it is not Aronszajn null.

However, cubes are Haar-null.  Therefore it is natural to ask whether there is a purely unrectifiable set in $\ell_2$ (or, more generally, in a Banach space $X$), that has finite $\H$-measure but for which the set of those $x^*$ for which the projection is positive, is not Haar-null. In this paper, we will answer this question affirmatively. Moreover, we will show that:

\begin{theorem}\label{thm}
In every separable Banach space $X$ there is a purely unrectifiable set $E$ that has finite $\H$-measure, but every projection of $E$ is of positive Lebesgue measure.
\end{theorem}

Note that if $X$ is not separable, then Theorem \ref{thm} fails, for an obvious reason: every set $E$ of finite $\H$-measure is separable, and for every $x^*$ with $ \overline\spa\, E \subset \ker x^*$ we of course have $x^*(E)=0$.

In order to prove Theorem \ref{thm} in a general separable Banach space, we will construct a sequence $x_n$ that, in some sense, ``behaves'' like an unconditional basis.  Then we will use this sequence $x_n$ as some sort of coordinate vectors when we construct the set $E$.

Let us make this more precise. Although not every separable Banach space admits a Schauder basis, see \cite{enflo}, it is well-known that every Banach space $X$ contains a {\it basic sequence}.  That is a sequence $x_1,x_2,\dots$ such that every $x\in \overline{\spa}\{x_1,x_2,\dots\}$ can be expressed in a unique way as $\sum_{n=1}^\infty c_nx_n$ with some coefficients $c_n$. A basic sequence is called an {\it unconditional basic sequence}, if all 
these sums $\sum_{n=1}^\infty c_nx_n$ converge unconditionally. By the Dvoretzky-Rogers lemma (see \cite{dvore}), every Banach space $X$ contains a sequence $y_1,y_2,\dots$ for which the series $\sum y_n$ converges unconditionally (moreover, for every sequence $(\alpha_n)\in\ell_2$ one can find an unconditionally convergent series with $\|y_n\|=|\alpha_n|$). 

There are Banach spaces that do not admit any unconditional basic sequences \cite{gowers-maurey}. However, we will show that the construction of a basic sequence and the proof of the Dvoretzky-Rogers lemma can easily be combined together to obtain a basic sequence such that $\sum_{n=1}^\infty c_nx_n$ converges unconditionally provided that $|c_n|\leq 1$ for every $n$.  The main difference between the properties of our sequence and an unconditional basic sequence is that we require the sums to converge unconditionally only when there is a bound on the coefficients $c_n$, and not for every convergent $\sum c_nx_n$. As an illustrative example, consider $L^p([0,1])$ for $1<p<2$.  The trigonometric system, $\{ e^{2\pi i n y}\}_{n \in \mathbb{Z}}$ is a basis but not an unconditional basis for $L^p([0,1])$.  The sequence $x_n:=2^{-n}e^{2\pi i n y}$ is still a conditional basis for $L^p([0,1])$, but 
$\sum_{n=1}^{\infty} c_nx_n$ 
converges unconditionally whenever $|c_n|\leq 1$.


After we constructed our sequence $x_n$, we will proceed by defining our purely unrectifiable set as the image of a function $f:[0,1]\rightarrow X$.  We construct our function, $f$, by defining a sequence of component functions $f_n:[0,1]\rightarrow \mathbb{R}$ and letting
$f:=\sum_n f_nx_n$.

The component functions $f_n$ we use will have a similar nature (but, in fact, they are much simpler) than the component functions used in De Pauw's construction in $\ell_2$. Using also the fact that $x_n$ is a carefully chosen sequence, we will show that the pure unrectifiability of the image of $f$ depends on the summability of the norms $\|x_n\|$. One of the main ingredients of our proof is Kirchheim's theorem \cite{kirchheim} for rectifiable metric spaces to compensate for the absence of the Radon-Nikodym property for general Banach spaces. This will enable us to show in every Banach space that the set $E$ we obtain is purely unrectifiable, provided that $\sum\|x_n\|$ does not converge. However, $\sum |x^*(x_n)|$ converges for any $x^*\in X^*$, therefore, from the ``point of view'' of any $x^*$, the $\sum\|x_n\|$ ``looks like'' a convergent series, which in turns means that $E$ ``looks like'' a rectifiable set. This will enable us to show that indeed the projections of $E$ has positive measure. 

We will make the above reasoning explicit by showing that $(x^*(x_n))\in\ell_1$ implies that we can factor $x^*$ through a projection to $\R^2$, the image of which is a rectifiable planar set that has many approximate tangents. Thus, the projection of this planar set onto every line will have positive measure.

Finally, we round up the paper by studying the converse direction in the Besi\-covitch-Federer projection theorem, namely, the projections of rectifiable sets in infinite dimensional spaces, in Section \ref{s:proof-and-rect}. Not surprisingly, we will see that finding large projections are much easier for rectifiable sets than for purely unrectifiable sets.


\bigskip
Throughout this paper $X$ will denote a fixed infinite dimensional separable Banach space with norm $\|\cdot\|$, and the unit sphere of $X$ is denoted by $S(X)=\{x:\,\|x\|=1\}$. For a sequence of real numbers $\alpha$, by $\|\alpha\|$ we denote the $\l_2$ norm of $\alpha$ and by $\|\alpha\|_1$ the $\l_1$ norm of $\alpha$. We will use $|\cdot |$ to signify both absolute value of numbers, and Lebesgue measure of sets on $\R$.


\section{Construction of the sequence $x_n$}\label{s:construct-xn}

In this section we prove the following:
\begin{proposition}\label{prop}
In every Banach space $X$, for every sequence $\alpha=(\alpha_n)\in\ell_2$ and for every $\eps>0$ there is a basic sequence $x_1,x_2\dots$ such that $\|x_n\|=|\alpha_n|$, and
\begin{itemize}
\item[(i)] $\sum_{n=1}^\infty c_nx_n$ converges whenever $|c_n|\le 1$ for all $n$,
\item[(ii)] $\sup_{|c_n|\le 1}  \|\sum_{n=1}^\infty c_nx_n\|\le(1+\eps)\|\alpha\|.$
\end{itemize}
\end{proposition}

Note that in (i) the series converges unconditionally since, for every series $\sum y_n$, unconditional convergence is equivalent to the fact that $\sum\eps_n y_n$ converges with any choice of $\eps_n=\pm 1$.

For the convenience of the reader first we recall the construction of the sequence in the proof of the Dvoretzky-Rogers lemma. Then we will show that this construction indeed can be done in such a way that the sequence obtained satisfies Proposition \ref{prop}.

Without loss of generality we can assume that $\eps \leq 1$. Let $c=(1+\eps)^{1/2}$. We divide $\N$ into finite intervals $A_m$, and by applying Dvoretzky's theorem, for each $m$ we choose a finite dimensional subspace $X_m=\spa\{v_n:\,n\in A_m\}$ for some vectors $v_n\in S(X)$ satisfying \begin{equation}\label{dvor}c^{-1}\sum_{n\in A_m} c_n^2\le\|\sum_{n\in A_m} c_nv_n\|^2\le c\sum_{n\in A_m} c_n^2\end{equation} with any choice of the coefficients $c_n$. Let $x_n=\alpha_nv_n$. Then
$$\|\sum_{n\in A_m} c_nx_n\|^2\le c\sum_{n\in A_m} c_n^2\alpha_n^2\le c\sum_{n\in A_m} \alpha_n^2$$ if $|c_n|\le 1$ for any $n$. Then for any finite set $A\subset\N$,
$$\|\sum_{n\in A} c_nx_n\|\le\sum_m\|\sum_{n\in A\cap A_m} c_nx_n\|\le c\sum_{\{m:\,A\cap A_m\neq\emptyset\}}(\sum_{n\in A_m} \alpha_n^2)^{1/2}.$$
Therefore, provided that we choose, as we may, the sets $A_m$ s.t. $\sum_m(\sum_{n\in A_m}\alpha_n^2)^{1/2}<c\|\alpha\|_2$, the series $\sum c_nx_n$ is Cauchy and indeed
$\|\sum c_nx_n\|\le (1+\eps)\|\alpha\|$.

In order to obtain a basic sequence, we need 
\begin{equation*}
\|\sum_{n=1}^N c_nx_n\|\le K\|\sum_{n=1}^M c_nx_n\|
\end{equation*}
for any $N<M$ with a constant $K$ that does not depend on $N,M$. Eqivalently, we denote $X_0=\{0\}$, and we need to find $K_1,K_2$ such that for every $m$, if $x\in\spa\{X_k:\,k\le m\}$ and $y,z$ are two vectors in $X_{m+1}$ spanned by two disjoint subsets of the vectors 
$\{v_n:\,n\in A_{m+1}\}$, then 
\begin{equation}\label{K1}
\|x+y\|\le K_1\|x+y+z\|,
\end{equation} and also if $v\in \spa\{X_{m+2},X_{m+3},\dots\}$ then 
$$\|x+y+z\|\le K_2\|x+y+z+v\|.$$
The second condition is equivalent to: for every $x\in\spa\{X_1,X_2,\dots,X_{m+1}\}$ and for every $y\in\spa\{X_{m+2},X_{m+3},\dots\}$,
\begin{equation}\label{K2}
\|x\|\le K_2\|x+y\|.
\end{equation}

We denote $\eps_m=\eps/2^m$ and choose a finite dimensional subspace $X^*_m\subset X^*$ such that for any $x\in \spa\{X_1,\dots,X_m\}$ there is an $x^*=x^*_m\in S(X^*_m)$ with $|x^*(x)|\ge (1-\eps_m)\|x\|$. Inductively, for each $m$, we may choose our subspaces $X_{m+1}\subset\bigcap_{x^*\in S(X^*_m)}\ker x^*$. Then for any $m$, $x\in \spa\{X_1,\dots,X_m\}$, $x^*=x^*_m$ and $y\in X_{m+1}$, we have $\|x+y\|\ge |x^*(x+y)|=|x^*(x)|\ge (1-\eps_m)\|x\|$. Iterating this, we can see that \eqref{K2} holds with $K_2:=\prod(1-\eps_m)^{-1}<\infty$.

If $y,z\in X_{m+1}$ are two vectors in $X_{m+1}$ spanned by disjoint vectors $\{v_n:\,n\in A_{m+1}\}$, then $\|y\|\leq c^2\|y+z\|=(1+\eps)\|y+z\|\le 2\|y+z\|$ by \eqref{dvor}. We also have $\|x+y+z\|\ge(1-\eps_m)\|x\|\ge \|x\|/2$, therefore 
\begin{equation}\notag \|x+y\| \leq \|x\|+\|y\|\le \|x\|+2\|y+z\|\le 
2\|x+y+z\|+3\|x\|\leq 8\|y+x+z\|.
\end{equation}
So indeed, \eqref{K1} holds, with $K_1=8$. This finishes the proof of Proposition \ref{prop}.

\section{Construction of the set $E$}\label{s:construction}

Let $X$ be a separable Banach space and $\alpha=(\alpha_n)$ be an arbitrary sequence in $\ell_2$ with $\|\alpha\|<1$.  We apply Proposition \ref{prop} with an $\eps$ small enough so that $(1+\eps)\|\alpha\|<1$ to obtain a sequence $x_1,x_2,\dots$. Also, we assume that $Y:=\overline{\spa}\{x_1,x_2,\dots\}\neq X$ and fix an arbitrary $x_0\in X\setminus Y$ with $\|x_0\|=1$. We will also use the notation $\alpha_0=\|x_0\|=1$. Then $x_0,x_1,\dots$ is a basic sequence for which 
\begin{equation}\label{K}
K:=\sup_{c_n\in[-1,1]}\|\sum_{n=0}^\infty c_nx_n\|<2.
\end{equation}

As usual, we denote by $x_n^*$ the linear functional that maps $x=\sum c_nx_n$ to $c_n$ for $x\in\ \overline{\spa}\{x_0,x_1,x_2,\dots\}:=Y_0$ and extend it to a linear functional $x_n^*\in X^*$ with the same norm. 
Then $\|x_n\|\|x_n^*\|=|\alpha_n|\|x_n^*\|\le b$ for some constant $b$ and for every $n$.

Let $g(t)$ denote the function 
$$g(t)= \begin{cases} 0 & {\rm if}\ t-[t]\in [0,\frac{1}{2})\\
t-[t]-1/2 & {\rm if}\ t-[t]\in [\frac{1}{2},1),\end{cases}$$ 
where $[t]\in \N$ is the integer part of $t\in\R$.

We also fix a sequence of even natural numbers $m_n$ that we will specify later, and put $M_0=1$, $M_n:=\prod_{k=1}^n m_k$ for $n\ge 1$. We denote the collection of all intervals of the form $I=[\frac{k-1}{M_n},\frac{k}{M_n})$, $k=1,2,\dots,M_n$, by $\II_n$. Let $f_0(t)=t$, and for $n\ge 1$ we define $$f_n(t)=M_n^{-1}g(M_n t)$$ 
and we define
$$f= \sum_{n=0}^\infty f_n x_n.$$

Since $|f_n(t)|\le 1$, by (i) of Proposition \ref{prop}, $f(t)\in X$
for every $t\in[0,1]$.  
We define our set $E$ as the image of the function $f:\,[0,1]\to X$
\[E = \{f(t): t \in [0,1]\}.\] 

There are two simple but essential properties that we will require in order to show that $E\subset X$ is purely unrectifiable.  The first one (Lemma \ref{2.2}) says that $f$ satisfies a certain Lipschitz type property. On the other hand, the second property (Lemma \ref{2.3}) will imply that we can still find large difference quotients in the neighbourhoods of almost every point in $E$.

\begin{lemma}\label{2.2} For any measurable set $S\subset [0,1]$,
$$
\H(\{f(t) : t\in S\}) \leq K\, |S|.$$
In particular, $\H(E)<2$, and $f$ satisfies Luzin's condition: $$\H(\{f(t) : t\in N\})=0$$ for any Lebesgue null set $N\subset[0,1]$. 
\end{lemma}

\begin{proof}
Note that, for $k\leq n$, each $f_k$ is Lipschitz on each interval $I\in\II_n$ with Lipschitz constant $1$, and for $k>n$ the function $f_k$ oscillates at most $1/2M_k<|I|$. Therefore for any $t,u\in I$ and for any $k$, $|f_k(t)-f_k(u)|\le |I|$ and $f(t)-f(u)$ can be expressed as $|I|\sum c_kx_k$ where $|c_k|=|f_k(t)-f_k(u)|/|I|\le 1$. By the definition of $K$ in \eqref{K}, $\diam f(I)\le K|I|$. 

The statement for a general measurable set $S$ follows by approximating $S$ by a countable union of dyadic intervals.
\end{proof}

The following lemma is a standard application of the Borel-Cantelli lemma:

\begin{lemma}\label{2.3}
If $(\alpha_n)\not\in\ell_1$ and $\alpha_n m_n\in\NN$ for every $n$, then for a.e. $t\in [0,1]$ there is an arbitrary large $n$ such that $d(t,M_n^{-1}\mathbb N)\le M_n^{-1}|\alpha_n|.$
\end{lemma}

\begin{proof}
One checks readily that the events $A_n:=\{d(t,M_n^{-1}\mathbb N)\le M_n^{-1}|\alpha_n|\}$ are independent and the probability of $A_n$ is $2|\alpha_n|$ (whenever
$2|\alpha_n|\le 1$). Indeed, each $I\in\II_{n-1}$ has $m_n$ subintervals in $\II_{n}$, and $d(t,M_n^{-1}\mathbb N)\le M_n^{-1}|\alpha_n|$ holds on $2|\alpha_n| m_n$ many out of these $m_n$ subintervals.
\end{proof}


\section{Rectifiability of the set $E$}\label{s:rect}
Next we study the rectifiability properties of $E$.  First we show the following:

\begin{proposition}\label{p:unrectifiable}
Suppose that $(\alpha_n)$ satisfies the requirements of Lemma \ref{2.3}. Then
$E$ is purely unrectifiable.
\end{proposition}

We will require the following lemma, which is a restatement of Kirchheim's theorem \cite{kirchheim}.

\begin{lemma}\label{l:kirchheim}
Suppose that $E\subset X$ and $\gamma \colon [0,1]\to X$ is Lipschitz with
\[\H(\gamma([0,1]) \cap E)>0.\]
Then there exist a measurable $A\subset [0,1]$ of positive measure and an $L \geq 1$ such that $\gamma(A):= F \subset E$, $\gamma$ restricted to $A$ is bi-Lipschitz with bi-Lipschitz constant $L$, and such that for $\H$-a.e. $y_0 \in F$ and every $\eps>0$, if $r$ is sufficiently small then 
\begin{equation}\label{density}
\H(F\cap B(y_0,r))\geq 2r(1-\varepsilon).
\end{equation}
\end{lemma}

As usual, a function is said to be \emph{bi-Lipschitz} with bi-Lipschitz constant $L$ if both the function and its inverse are Lipschitz with Lipschitz constant $L$.

\begin{proof}[Proof of Proposition \ref{p:unrectifiable}]
Suppose that $E$ is not purely unrectifiable and suppose that $\gamma\colon [0,1] \to X$ satisfies the conclusion of Lemma \ref{l:kirchheim}.  We will use the same notation as in Lemma \ref{l:kirchheim}.

Note that, since $\gamma$ is bi-Lipschitz on $A$, if we write $y_0=\gamma(s_0)$, then ``$\H$-a.e. $y_0\in F$'' is equivalent to ``$\H$-a.e. $s_0\in A$''. Also if for each $y_0$ we choose a $t_0$ such that $y_0=f(t_0)$, then by Lemma \ref{2.3}, for almost every $t_0$ there are infinitely many $n$ for which $d(t_0,M_n^{-1}\N)<M_n^{-1}|\alpha_n|$. By the Luzin property of $f$, this also implies that for $\H$-a.e. $y_0\in F$, there are infinitely many $n$ such that
\begin{equation}\label{t}\dist(t_0,I^c)<M_n^{-1}|\alpha_n|,\end{equation}
where $I$ is the interval in $\II_n$ that contains $t_0$.  In what follows, we fix a $y_0=\gamma(s_0)=f(t_0)$ for which \eqref{density} holds for every $\epsilon>0$ and sufficiently small $r>0$, and also \eqref{t} holds for infinitely many $n$. We also assume that $s_0$ is a density point of $A$. 

Recall that $K<2$. Therefore, we can fix some positive constants $\lambda$ and $\eps$ satisfying the inequalities $(\lambda+1)/\lambda<2/K$ and $\eps<1-K(\lambda+1)/2\lambda$.
Note that by Lemma \ref{2.2} and by \eqref{density}, for every sufficiently small $r$ we must have
\begin{equation}\label{e:notisolated}|\{t\in [0,1]:f(t)\in F\cap B(y_0,r)\}|\geq \frac{2r(1-\varepsilon)}{K}>\frac{r(\lambda+1)}{\lambda}.\end{equation} 
By putting $r=\lambda M_n^{-1}|\alpha_n|$ into \eqref{e:notisolated}, the right hand side equals $r+M_n^{-1}|\alpha_n|$, therefore, by \eqref{t}, there exists a $t_n\in [0,1]$ that does not belong to the same interval $I\in\II_n$ as $t_0$, and such that $f(t_n)\in F\cap B(y_0,r)$. We let $s_n\in A$ be such that $f(t_n)=\gamma(s_n)$ and we write $[s_0,s_n]$ for the non-trivial closed interval with endpoints $s_0,s_n$ (we do not assume that $s_0<s_n$).  Note that, since $\gamma$ is bi-Lipschitz,
\begin{equation}\label{e:sdist}|s_0-s_n| \leq L\|\gamma(s_0)-\gamma(s_n)\|\leq Lr = L\lambda M_n^{-1}|\alpha_n| \to 0.\end{equation}

By the definition of the functions $f_n$, for any $n\ge 1$, one of the values of $f_n(t_0),f_n(t_n)$ is zero and the other is at least $1/2M_n-r=(1/2-\lambda|\alpha_n|)/M_n>1/4M_n$ for every large enough $n$.  Therefore, there exist arbitrarily large $n$ for which the length of the interval $x_n^*(\gamma([s_0,s_n]))$ is at least
\begin{equation}\label{e:gamma-big-meas}|x_n^*(\gamma([s_0,s_n])|\geq |f_n(t_n)-f_n(t_0)| \geq 1/4M_n.\end{equation} 

Further, recall that $s_0$ is a density point of $A$ and that $|\alpha_n|\|x_n^*\|<b$.  Therefore, if we fix $0<\delta < (16 L^2\lambda b)^{-1}$, there exists an $R>0$ such that, if $0<\rho<R$,
\[|[s_0-\rho,s_0+\rho]\cap A| \geq (1-\delta)2\rho.\]
Since $s_n\to s_0$, for large enough $n$ we see that this inequality is true for $\rho=|s_0-s_n|$.  Then, since $\gamma$ is $L$-Lipschitz, we use \eqref{e:sdist} to deduce
\[\H(\gamma([s_0,s_n]) \setminus F)\le\H(\gamma([s_0,s_n]\setminus A)) \leq 2L\delta\rho \leq 2L^2\delta\lambda M_n^{-1} |\alpha_n|.\]
Since $x_n^*$ is $b/|\alpha_n|$ Lipschitz,
\[|x_n^*(\gamma([s_0,s_n])\setminus F)| \leq 2L^2\delta\lambda M_n^{-1}b.\]
Combining this with \eqref{e:gamma-big-meas} and using the choice of $\delta$, we see that there are arbitrarily large $n$ for which
\begin{equation}\label{e:gamma-cont}|x_n^*(\gamma([s_0,s_n]\cap F)| \geq 1/8M_n.\end{equation}

On the other hand, for any $n$,
\begin{align*}
|x_n^*(\gamma([s_0,s_n]) \cap F)| &\leq |f_n(\{t: f(t)\in \gamma([s_0,s_n])\}| \\
&\leq |f_n(\{t: t\in x_0^*(\gamma ([s_0,s_n]))\}|\\
&\leq |f_n(\{t: |t-t_0| \le \|x_0^*\| L |s_n-s_0|\})|\\
&\leq 2 b L^2 \lambda M_n^{-1}|\alpha_n|.
\end{align*}
Note that, in the final inequality we have used the fact that $f_n$ is piecewise Lipschitz with Lipschitz constant 1, \eqref{e:sdist} and the fact that $\|x_0^*\|\leq b$.  Since $(\alpha_n)\in \l^2$ we have $\alpha_n\to 0$.  This contradicts \eqref{e:gamma-cont}.
\end{proof}

Proposition \ref{p:unrectifiable} is nicely complemented by the following proposition:
\begin{proposition}\label{p:rect}
Suppose that $\alpha\in \l_1$ and $\sum_{n=1}^\infty|\alpha_n|<1$. Then $E$ is rectifiable.
\end{proposition}

\begin{proof}
We will in fact show that $E$ is contained within a curve of finite length.  Without loss of generality we can assume that $X=\ell_1$, and $x_n=\alpha_ne_n$ where $e_0,e_1,\dots$ is the standard basis of $\ell_1$. Indeed, the mapping $T:\,\ell_1\to X$ defined by $\alpha_ne_n\to x_n$ is Lipschitz with Lipschitz constant 1, therefore it cannot increase the length of any curve.

We will say that a line segment is \emph{vertical} if $e_0^*$ is constant on it.
For each $n$ let $$E_n=\{\sum_{k=0}^n f_k(t)x_k:\,t\in[0,1]\}.$$ Since the functions $f_n$ are linear on each interval $I\in\II_{n+1}$, and continuous on each $I\in\II_n$, therefore $E_n$ consists of $M_n$ pieces, each of which is a polygon consisting of two non-vertical line segments, and when projected onto $e_0$, each piece is mapped onto an interval $I\in\II_n$. We define $\gamma_n$ to be the polygon that connects the right endpoint of each piece of $E_n$ to the left endpoint of the next piece of $E_n$ by the (vertical) line segment between these two points, in the natural order.

To calculate the length of $\gamma_n\subset\ell_1$, we can simply add together the length of the projections of its line segments onto the coordinate directions $e_0,\dots,e_n$. In other words, we need to calculate the sum of the lengths of $\ga_n-\ga_{n-1}$ in $\ell_1$. This is very easy: since $f_n$ is piecewise Lipschitz with Lipschitz constant 1, the $f_n(t)x_n$ term adds at most $|\alpha_n||I|$ length on every interval $I$; and the vertical segments add at most $|\alpha_n|/2M_n$ length at the right endpoint of the intervals $I\in\II_n$. Therefore the length of each $\gamma_n-\gamma_{n-1}$ is at most $\frac{3}{2}|\alpha_n|$.

There is a natural common parametrization of the curves $\ga_n$ by the interval $[0,1]$. Let $\tau:\,[0,1]\to[0,1]$ be a continuous increasing function for which the preimage of each endpoint of each interval $I\in\II_n$ is a non-degenerate interval. Now choose the (unique) parametrization of $\ga_n$ for which $e_0^*(\ga_n(t))=\tau(t)$, and for which $\ga_n(t)$ is linear on each interval on which $\tau$ is constant. With this parametrization, $\|\ga_n-\ga_{n-1}\|_\infty\le\frac 32|\alpha_n|$, therefore $\ga_n$ converges uniformly to a continuous $\ga:\,[0,1]\to \ell_1$ whose image has length at most $\frac{3}{2}$, and it covers the set $E$. Indeed, $f(t)\in\ga(\tau^{-1}(t))$ for every $t\in[0,1]$.


\end{proof}

\section{Projections of the set $E$}\label{s:projection}
As before, we denote $Y_0=\overline{\spa}\{x_0,x_1,x_2,\dots\}$.  In this section we show that for any $0\neq x^*\in Y_0^*$, $x^*$ maps $E$ onto a set of positive measure. Without loss of generality we can assume that $\|x^*\|=1$. We fix such an $x^*$ and denote $x^*(x_n)=\tilde\alpha_n$. By choosing $\eps_n=\pm 1$ to be the sign of $\tilde\alpha_n$, it follows from (ii) of Proposition \ref{prop} that
\[\sum_{n=1}^\infty |\tilde\alpha_n| = \sum_{n=1}^\infty \eps_n x^*(x_n) = x^*(\sum_{n=1}^\infty \eps_n x_n) <1.\]

Consider our construction of a function and its graph (call them $\tilde f$, $\tilde E$) in $\ell_1$ with $\alpha_n$ replaced by $\tilde\alpha_n$ for $n\ge 1$, and with $\tilde x_0=e_0$, $\tilde x_n=\tilde\alpha_n e_n$ where $e_0,e_1,\dots$ is the standard basis in $\ell_1$. By Proposition \ref{p:rect}, $\tilde E$ is rectifiable.

Now consider the projection $P:\,\ell_1\to\R^2$ defined by $e_0\to x$, $e_n\to y$ for $n\ge 1$, where $x,y$ are the standard coordinates of $\R^2$. Then
\begin{equation}\label{e:projection-plane}P(\tilde f(t))=(t,\sum_{n=1}^\infty f_n(t)\tilde \alpha_n).\end{equation}
The set $P(\tilde E)\subset\R^2$ is rectifiable, since it is covered by the Lipschitz image of the rectifiable set $\tilde E$. Also note that the projection of $P(\tilde E)$ to the $x$ coordinate direction is the whole interval $[0,1]$ and it maps $P(\tilde f(t))$ to $t$.  Therefore, for Lebesgue positively many $t\in[0,1]$, $P(\tilde E)$ has an approximate tangent at $P(\tilde f(t))$ and this tangent is not vertical.

Fix such a $t\in[0,1]$. Then $P(\tilde E)$ has positive projection onto every line except possibly to the line orthogonal to the approximate tangent
at $P(\tilde f(t))$. Let $s$ denote the slope of this approximate tangent.

If $\tilde\alpha_n=0$ for every $n\ge 1$ then $\tilde\alpha_0\neq 0$ and $x^*$ maps $E$ onto an interval of length $|\tilde\alpha_0|> 0$. From now on we assume that $\tilde\alpha_n\neq 0$ for some $n\ge 1$ and we fix such an $n$.

Let $I$ denote the interval $I\in\II_n$ that contains $t$. Choose $\pm$ so that $t':=t\pm 1/2M_n$ also belongs to $I$. Since the functions $f_m$ are linear on $I$ for $m<n,$ and they are periodic with period $1/M_m$ and hence also with period $1/2M_n$ for $m>n$, therefore 
\[f_m(t'+h)-f_m(t') = f_m(t+h)-f_m(t)\] for any $m\neq n$ and $t,t',t+h,t'+h\in I$. On the other hand, for $m=n$ and for $t,t',t+h,t'+h\in I$, one of the two sides of the equation is 0 and the other is $h\neq 0$ (this is where we make use of the sloped line segments in the definition of each $f_n$). Therefore $P(\tilde E)$ also has an approximate tangent at $P(\tilde f(t'))$ of slope $s'=s\pm\tilde\alpha_n\neq s$, and so $P(\tilde E)$ has a positive projection in every direction.

The proof is finished by considering the projection of $P(\tilde E)$ in the $(\tilde\alpha_0,1)$ direction: by \eqref{e:projection-plane}, for every $t\in [0,1]$ we obtain
$$\tilde\alpha_0 t+1\sum_{n=1}^\infty f_n(t)\tilde \alpha_n=\sum_{n=0}^\infty f_n(t)\tilde \alpha_n=x^*(f(t)),$$
so indeed, $|x^*(E)|>0$.

\section{Conclusion}\label{s:proof-and-rect}
Let $\al=(\al_n)$ be an arbitrary sequence of rational numbers in $\ell_2\setminus\ell_1$ with $\|\al\|<1$, and fix some even numbers $m_n$ s.t. $\alpha_nm_n\in\N$ for every $n$. We construct the sequence $x_n$ as in Section \ref{s:construction}, denote $Y=\overline{\spa}\{x_1,x_2,\dots\}$, and fix a dense sequence $x_0^n$, $n=1,2,\dots$ of $X\setminus Y$. Then, for each $x_0=x_0^n$ we define the set $E=E^n$ as in Section \ref{s:construction}.  By Lemma \ref{2.2}, $\H(E^n) < 2$ and each $E^n$ is purely unrectifiable by Proposition \ref{p:unrectifiable}.

Now suppose that $0\neq x^*\in X^*$.  Then there exists some $x_0^n$ such that $x^*$ is not identically zero on $Y_0^n:=\overline\spa\{x_0^n,x_1,x_2,\ldots\}$.  Thus, by Section \ref{s:projection}, $|x^*(E^n)|>0$.  Therefore
\[E=\bigcup_{n=1}^\infty \frac{1}{2^n}E^n\]
satisfies all the requirements of Theorem \ref{thm}. 

\bigskip
We finish this paper by a brief discussion of the projections of a rectifiable set. Recall that for a rectifiable set $E\subset\R^n$ of positive measure, all projections are positive except possibly those that belong to a 1-codimensional subspace, namely, the ones for which $\theta\in \ker x^*$ for an approximate tangent $\theta$. Although in a general Banach space we may not have any approximate tangents, we show that there is always at most a 1-codimensional subspace of zero projections:

\begin{theorem}\label{2} For every rectifiable set $E\subset X$ of positive $\H$-measure, $x^*(E)$ has positive Lebesgue measure for every $x^*\in X^*$ except possibly for those that belong to a closed linear subspace $Y^*\neq X^*$.
\end{theorem}

Of course, this theorem is sharp in every Banach space, e.g. if the set $E$ is a straight line then indeed it has positive projection in all directions except for those that belong to a 1-codimensional subspace.

\begin{proof}
By Lemma \ref{l:kirchheim}, there is a Lipschitz curve $\gamma:\,[0,1]\to X$ such that $\gamma$ is bi-Lipschitz with a bi-Lipschitz constant $L$ on a set $A\subset[0,1]$ of positive measure, and $\ga(A)\subset E$. Our aim is to show that $\gamma(A)$ has large projections.

We pick $s_1,s_2\in A$ such that $|A\cap[s_1,s_2]|\ge (1-1/2L^2)|s_2-s_1|$. Then, for any $x^*\in S(X^*)$ for which 
$|x^*(\ga(s_2)-\ga(s_1))|>\|\ga(s_2)-\ga(s_1)\|/2$,
we obtain
\begin{equation}\label{rect}|x^*(\gamma(A))|>\|\ga(s_2)-\ga(s_1)\|/2-|x^*(\gamma([s_1,s_2]\setminus A))|
\end{equation} where $\|\ga(s_2)-\ga(s_1)\|/2\ge |s_2-s_1|/2L$ and
$|x^*(\gamma([s_1,s_2]\setminus A))|\le L|[s_1,s_2]\setminus A)|\le |s_2-s_1|/2L$. That is, the left side of \eqref{rect} is positive and so we have found at least one positive projection for $\gamma(A)$.

Now consider an arbitrary $x^*\in X^*$. Since $x^*\circ\gamma$ is Lipschitz, it has a derivative $b=b_{x^*}\in L^{\infty}(A)$ bounded by $\|x^*\|L$. We denote by $T$ the bounded linear mapping $X^*\to L^\infty(A)$ defined by 
$x^*\to b_{x^*}$. Since a Lipschitz function $\R\to\R$ maps $A$ onto a null set if and only if its derivative is zero at almost every point of $A$, $x^*(\gamma(A))$ has zero measure if and only if $x^*\in\ker T$. This is a closed linear subspace and from the previous paragraph we know that indeed $\ker T\neq X^*$.
\end{proof}
\


\begin{thebibliography}{7}

\bibitem{b-lindenstrauss}{
Y.~ Benyamini and J.~ Lindenstrauss,
\textit{Geometric Nonlinear Functional Analysis},
Colloq. Publ. Amer. Math. Soc., Rhode Island, (2000).}

\bibitem{depauw}{
T.~De Pauw,
\textit{An example pertaining to the failure of the Besicovitch-Federer Theorem in separable Hilbert space},
Publ. Mat., to appear.}

\bibitem{dvore}{
A. Dvoretzky and C. A. Rogers,
\textit{Absolute and unconditional convergence in normed linear spaces},
Proc. Natl. Acad. Sci. USA, {\bf 36} No. 3 (1950 mar), 192-197.}

\bibitem{enflo}{
P.~Enflo,
\textit{A counterexample to the approximation problem in Banach spaces},
Acta Math., {\bf 130} (1973), 309--317.}

\bibitem{gowers-maurey}{
W.T.~Gowers and B.~Maurey, 
\textit{The unconditional basic sequence problem},
J. Amer. Math. Soc., {\bf 6} (1993), No. 4, 851--874.}

\bibitem{kirchheim}{
B.~Kirchheim, 
\textit{Rectifiable metric spaces: local structure and regularity of the Hausdorff measure},
Proc. Amer. Math. Soc., {\bf 121} (1994), 113--123.}

\bibitem{mattila}{
P.~Mattila,
\textit{Geometry of sets and measures in Euclidean spaces},
Cambridge Univ. Press, (1995).}


\end{thebibliography}
\end{document}